\documentclass[12pt,reqno]{amsart}

\usepackage{a4wide}

\pdfoutput=1

\usepackage[utf8]{inputenc}
\usepackage[english]{babel}
\usepackage{amsmath,amsfonts,amsthm,amssymb,amscd,amsbsy}
\usepackage[all]{xy}
\usepackage{graphicx}
\usepackage{euscript}
\usepackage{mathtext}
\usepackage{upgreek}
\usepackage{hyperref}
\hypersetup{
    colorlinks=true,
    linkcolor=blue,
    filecolor=magenta,      
    urlcolor=cyan,
}
\usepackage{euscript}
\usepackage{mathtools}
\usepackage{microtype}
\usepackage{setspace}
\usepackage{fancyhdr}
\usepackage{pict2e}
\usepackage{mathrsfs}
\usepackage[shortlabels]{enumitem}
\usepackage{stackengine}
\usepackage{changepage}
\usepackage{parskip}
\usepackage{tikz}
\usetikzlibrary{arrows}
\usepackage{lmodern}
\usepackage{csquotes}
\usepackage{chngcntr}
\counterwithout{equation}{section}

\usepackage{xpatch}
\xapptocmd\normalsize{%
 \abovedisplayskip=10pt plus 1pt minus 3pt
 \abovedisplayshortskip=1pt plus 3pt
 \belowdisplayskip=10pt plus 2pt minus 3pt
 \belowdisplayshortskip=8pt plus 3pt minus 2pt
}{}{}

\numberwithin{equation}{section}

\usepackage[
backend=biber,
style=ieee,
sorting=nyt,
giveninits=true,
dashed=false,
maxbibnames=99
]{biblatex}
\DeclareSymbolFont{largesymbols}{OMX}{cmex}{m}{n}

\newtheoremstyle{theoremstyle}
  {0mm} 
  {0mm} 
  {\itshape} 
  {} 
  {\bfseries} 
  {.} 
  {.5em} 
  {} 

\theoremstyle{theoremstyle}
\newtheorem{proposition}{Proposition}[section]
\newtheorem{lemma}[proposition]{Lemma}
\newtheorem{theorem}[proposition]{Theorem}
\newtheorem*{maintheorem}{Main Theorem}

\newtheorem{corollary}{Corollary}[proposition]

\newtheoremstyle{examplestyle}
  {0mm} 
  {0mm} 
  {} 
  {} 
  {\bfseries} 
  {.} 
  {.5em} 
  {} 

\theoremstyle{examplestyle}
\newtheorem{remark}[proposition]{Remark}

\newtheoremstyle{remarks}
  {0mm} 
  {0mm} 
  {} 
  {} 
  {\bfseries} 
  {.} 
  {-.5em} 
  {} 

\theoremstyle{remarks}

\renewcommand{\Im}{\hspace{0.08em}\mathrm{Im}\hspace{0.04em}}
\newcommand{\Ker}{\mathrm{Ker} \hspace{0.04em}}
\newcommand{\ccdot}{\hspace{-1.2pt} \cdot}

\newcommand{\rank}{\mathrm{rank} \hspace{0.1em}}

\newcommand{\codim}{\mathrm{codim}}

\newcommand{\Aut}{\mathrm{Aut} \hspace{0.06em}}

\newcommand{\End}{\mathrm{End}}

\newcommand{\DD}{\mathrm{DD}}
\newcommand{\Lie}{\mathrm{Lie} \hspace{0.06em}}

\newcommand{\defeq}{\vcentcolon=}

\newcommand{\Rl}{\mathbb{R}}
\newcommand{\Cx}{\mathbb{C}}
\newcommand{\Hq}{\mathbb{H}}
\newcommand{\Oo}{\mathbb{O}}
\newcommand{\Ad}{\mathrm{Ad} \hspace{0.06em}}
\newcommand{\ad}{\mathrm{ad} \hspace{0.06em}}
\newcommand{\mk}[1]{\mathfrak{{#1}}}
\newcommand{\mb}[1]{\mathbb{{#1}}}
\newcommand{\mr}[1]{\mathrm{{#1}}}
\newcommand{\mc}[1]{\mathcal{{#1}}}

\newcommand{\set}[1]{\hspace{-0.8pt} \left \{ \hspace{0.03em} {#1} \hspace{0.03em} \right \}}

\newcommand{\cross}[2]{\langle \hspace{0.1em} {#1} \hspace{0.1em} | \hspace{0.1em} {#2} \hspace{0.1em} \rangle \hspace{0.02em}}
\newcommand{\restr}[2]{{\left.\kern-\nulldelimiterspace #1 \vphantom{\big|} \right|_{#2}}}

\renewcommand{\O}{\mathrm{O}}

\newcommand\altxrightarrow[2][0pt]{\mathrel{\ensurestackMath{\stackengine%
  {\dimexpr#1-7.5pt}{\xrightarrow{\phantom{#2}}}{\scriptstyle\!#2\,}%
  {O}{c}{F}{F}{S}}}}
\newcommand{\isoto}{\altxrightarrow[1pt]{\sim}}

\newcommand{\mysetminusD}{\hbox{\tikz{\useasboundingbox (-0.5pt,-0.5pt) rectangle (5pt,8pt); \draw[line width=0.6pt,line cap=round] (3.5pt,-1.5pt) -- (0,7.25pt);}} \hspace{-1pt}}
\newcommand{\mysetminusT}{\mysetminusD}
\newcommand{\mysetminusS}{\hbox{\tikz{\draw[line width=0.45pt,line cap=round] (2pt,0) -- (-0.5pt,5pt);}} \hspace{1pt}}
\newcommand{\mysetminusSS}{\hbox{\tikz{\draw[line width=0.4pt,line cap=round] (1.5pt,0) -- (0,3pt);}}}

\newcommand{\mysetminus}{\mathbin{\mathchoice{\mysetminusD}{\mysetminusT}{\mysetminusS}{\mysetminusSS}}}
 
\makeatletter
\newcommand{\extp}{\@ifnextchar^\@extp{\@extp^{\,}}}
\def\@extp^#1{\mathop{\bigwedge\nolimits^{\!#1}}}
\makeatother

\makeatletter
\DeclareRobustCommand{\loplus}{\mathbin{\mathpalette\dog@lsemi{+}}}
\DeclareRobustCommand{\roplus}{\mathbin{\mathpalette\dog@rsemi{+}}}

\newcommand{\dog@rsemi}[2]{\dog@semi{#1}{#2}{-90,90}}
\newcommand{\dog@lsemi}[2]{\dog@semi{#1}{#2}{270,90}}
\newcommand{\dog@semi}[3]{%
  \begingroup
  \sbox\z@{$\m@th#1#2$}%
  \setlength{\unitlength}{\dimexpr\ht\z@+\dp\z@\relax}%
  \makebox[\wd\z@]{\raisebox{-\dp\z@}{%
    \begin{picture}(1,1)
    \linethickness{\variable@rule{#1}}
    \roundcap
    \put(0.5,0.5){\makebox(0,0){\raisebox{\dp\z@}{$\m@th#1#2$}}}
    \put(0.5,0.5){\arc[#3]{0.5}}
    \end{picture}%
  }}%
  \endgroup
}
\newcommand{\variable@rule}[1]{%
  \fontdimen8  
  \ifx#1\displaystyle\textfont3\else
    \ifx#1\textstyle\textfont3\else
      \ifx#1\scriptstyle\scriptfont3\else
        \scriptscriptfont3\relax
  \fi\fi\fi
}
\makeatother

\begin{document}

\title[Homogeneous codimension one foliations]{Homogeneous codimension one foliations on reducible symmetric spaces of noncompact type}
\author{Ivan Solonenko}
\address{Department of Mathematics, King's College London, United Kingdom}
\email{ivan.solonenko@kcl.ac.uk}

\begin{abstract}
We extend the classification of homogeneous codimension-one foliations on irreducible Riemannian symmetric spaces of noncompact type obtained by Berndt and Tamaru in \cite{berndttamarufoliations} to the reducible case, thus completing it for all noncompact symmetric spaces.
\end{abstract}

\maketitle

\vspace{-1.5em}

\counterwithout{equation}{section}

\section{Introduction}

A foliation $\mc{F}$ with connected properly embedded leaves on a connected Riemannian manifold $M$ is called homogeneous if the subgroup of isometries of $M$ preserving the leaves of $\mc{F}$ acts transitively on each leaf. Given a connected Lie group $G$ acting properly and isometrically on $M$ without singular orbits, the orbit foliation of $G$ is homogeneous, and, by design, every homogeneous foliation arises in this way. We say that $\mc{F}$ is polar (resp., hyperpolar) if it is the orbit foliation of a polar (resp., hyperpolar) action on $M$. Perhaps the easiest class of homogeneous foliations consists of foliations of codimension one. These correspond to cohomogeneity-one actions and are always hyperpolar. Thus, homogeneous codimension-one foliations are an interesting object of study in the intersection of the theories of polar actions and Riemannian foliations.

In \cite{berndttamarufoliations}, Berndt and Tamaru invented two new methods for constructing homogeneous codimension-one foliations on symmetric spaces of noncompact type. One of them produces such a foliation -- denoted by $\mc{F}_\ell$ -- for every one-dimensional linear subspace $\ell \subseteq \mk{a}$, where $\mk{a}$ is a maximal abelian subspace in the $(-1)$-eigenspace of a Cartan involution of the Lie algebra $\mk{g}$ of the isometry group of $M$. The other method produces a foliation $\mc{F}_{\upalpha_i}$ for each of the simple restricted roots $\upalpha_i \in \mk{a}^*$. These two types of foliations are parametrized by $\mb{P}\mk{a} \simeq \Rl P^{r-1}$ (here $r = \rank(M) = \dim (\mk{a})$) and $\set{\upalpha_1, \ldots, \upalpha_r}$, respectively. Restricting to the case when $M$ is irreducible, the authors proved that every homogeneous codimension-one foliation on it arises in one of these two ways -- up to isometric congruence. Moreover, if we let $\DD_M$ denote the Dynkin diagram of $M$, then its automorphism group $\Aut(\DD_M)$ acts naturally on $\set{\upalpha_1, \ldots, \upalpha_r}$ and $\mk{a}$, and two foliations $\mc{F}_\ell$ and $\mc{F}_{\ell'}$ (resp., $\mc{F}_{\upalpha_i}$ and $\mc{F}_{\upalpha_j}$) are congruent if and only if $\ell$ and $\ell'$ (resp., $\upalpha_i$ and $\upalpha_j$) lie in the same $\Aut(\DD_M)$-orbit (whereas the foliations $\mc{F}_\ell$ and $\mc{F}_{\upalpha_i}$ are never congruent). This gives the complete classification of homogeneous codimension-one foliations on irreducible symmetric spaces of noncompact type and shows that their moduli space is isomorphic to 
$$
(\mathbb{R}P^{r-1} \sqcup \set{1, \ldots, r})/\Aut(\mathrm{DD_M}).
$$
In particular, if $M$ is of rank one and thus isometric to a hyperbolic space over $\Rl, \Cx, \Hq,$ or $\Oo$, it has only two homogeneous codimension-one foliations up to congruence: $\mc{F}_\mk{a}$ and $\mc{F}_{\upalpha_1}$.

In their subsequent paper \cite{hyperpolarfoliations} with D\'{i}az-Ramos, Berndt and Tamaru came up with a geometric construction that allows one to produce homogeneous hyperpolar foliations on any symmetric space of noncompact type from the foliations $\mc{F}_{\upalpha_1}$ on hyperbolic spaces and foliations by parallel affine subspaces on the Euclidean spaces. They went on to show that every homogeneous hyperpolar foliation on $M$ arises via this construction up to congruence. Notably, they do not impose the irreducibility assumption on $M$. On the other hand, it is not generally known when two foliations obtained by this construction are congruent to each other. When applied to foliations of codimension one, their method produces precisely the foliations $\mc{F}_\ell \, (\ell \subseteq \mk{a})$ and $\mc{F}_{\upalpha_i} \, (1 \leqslant i \leqslant r)$ defined above. Altogether, this implies that on any symmetric space of noncompact type, any homogeneous codimension-one foliation is congruent to either some $\mc{F}_\ell$ or some $\mc{F}_{\upalpha_i}$. Hence, to complete the classification of such foliations, one just needs to tell when two foliations of the form $\mc{F}_\ell$ (resp., $\mc{F}_{\upalpha_i}$) are congruent to each other.

In this paper, we accomplish this by showing that the classification in the irreducible case works verbatim in the situation when $M$ is reducible:

\begin{maintheorem}
Let $M$ be a symmetric space of noncompact type and rank $r$. Then, given one-dimensional linear subspaces $\ell, \ell' \subseteq \mk{a}$ (resp., simple roots $\upalpha_i, \upalpha_j \in \set{\upalpha_1, \ldots, \upalpha_r}$), the foliations $\mc{F}_\ell$ and $\mc{F}_{\ell'}$ (resp., $\mc{F}_{\upalpha_i}$ and $\mc{F}_{\upalpha_j}$) are isometrically congruent if and only if $\ell$ and $\ell'$ (resp., $\upalpha_i$ and $\upalpha_j$) differ by some $P \in \Aut(\DD_M)$. Consequently, the moduli space of homogeneous codimension-one foliations on $M$ is isomorphic to
\vspace{-0.2em}$$
(\mathbb{R}P^{r-1} \sqcup \set{1, \ldots, r})/\Aut(\mathrm{DD_M}).
$$
\end{maintheorem}

\vspace{-1em}

One noteworthy caveat is that we make the routine assumption that the Riemannian metric of $M$ is normalized in such a way that it agrees with the Killing form of the Lie algebra of the isometry group. When $M$ is irreducible, this assumption is relatively benign, for it amounts to simply rescaling the Riemannian metric by a constant factor. It poses a nontrivial restriction, however, when $M$ is reducible, because different de Rham factors of $M$ may carry different normalization constants in general. Since the authors of \cite{hyperpolarfoliations} make this assumption as well, it remains an open question how much of the classification remains valid in the reducible case when the Riemannian metric is normalized in a different way.

The paper is organized as follows. In Section \ref{prerequisites} we review some rudiments of the theories of symmetric spaces of noncompact type and parabolic subgroups that will be necessary in the rest of paper. We also prove a somewhat folklore result about the isometry group of a Riemannian manifold decomposed as a Riemannian product. In Section \ref{main_section} we first introduce the results of Berndt, Tamaru, and D\'{i}az-Ramos and then prove the Main Theorem. The paper is concluded by Section \ref{appendix}, in which we give an alternative -- more algebraic -- proof of the uniqueness result for foliations of the type $\mc{F}_\ell$.

\textbf{Acknowledgments.} I am grateful to my supervisor J\"{u}rgen Berndt for his careful guidance and countless invaluable discussions on the topic.

\vspace{0.05em}

\section{Prerequisites}\label{prerequisites}

The purpose of this section is to establish the notation and terminology necessary to understand the classification results of Berndt, Tamaru, and D\'{i}az-Ramos as well as the rest of the paper. We also prove a useful result pertaining to the isometry group of a Riemannian manifold decomposed as a Riemannian product. We follow the notation established in \cite{berndttamarufoliations} and \cite{hyperpolarfoliations} and refer to the same articles as well as \cite{knapp, helgason, kobayashi_nomizu_I} for more detailed expositions.

\subsection{Polar and codimension-one foliations}

Let $M$ be a connected Riemannian manifold, and let $\mc{F}$ be a foliation on $M$ with connected properly embedded leaves. Then we call $\mc{F}$ \textit{homogeneous} if the group $I(M, \mc{F}) = \set{\upvarphi \in I(M) \mid \upvarphi(F) = F \hspace{0.35em} \text{for every leaf} \hspace{0.35em} F \in \mc{F}}$ acts transitively on each leaf of $\mc{F}$. Since the leaves are properly embedded, it follows that $I(M, \mc{F})$ is a closed Lie subgroup of $I(M)$. Moreover, its identity component $I^0(M, \mc{F})$ also acts transitively on each leaf because the leaves are connected. Conversely, given a connected Lie group $G$ acting properly and isometrically on $M$ without singular orbits, the leaves of the action form a homogeneous foliation. Note that the leaves of a homogeneous foliation are equidistant from each other.

Assume that $M$ is complete. We say that a homogeneous foliations is \textit{polar} if it is the orbit foliation of a polar action on $M$. In other words, $\mc{F}$ is polar if there exists a complete immersed submanifold $\Upsigma \subseteq M$, called a section, that intersects all the leaves of $\mc{F}$ and does so orthogonally. If exist, all sections are isometrically congruent to each other and necessarily totally geodesic. Many authors usually assume sections to be embedded or even properly embedded. A homogeneous polar foliation is called \textit{hyperpolar} if its sections are flat. Homogeneous foliations of codimension one turn out to be automatically hyperpolar. Indeed, given such a foliation $\mc{F}$, its sections should be just geodesics normal to its leaves. One can show that any normal geodesic intersects all the leaves, does so orthogonally, and its image is an immersed submanifold, hence a section.

\subsection{The Iwasawa and horospherical decompositions}\label{IHdecompositions}

Let $M$ be a symmetric space of noncompact type. We denote $G = I^0(M), \widetilde{G} = I(M),$ and $\mk{g} = \Lie(G)$. Pick any point $o \in M$ and write $\widetilde{K} = \widetilde{G}_o$ and $K = G_o = \widetilde{K}^0$ for the isotropy groups at $o$ and $\mk{k} = \Lie(K)$ for the isotropy Lie subalgebra of $\mk{g}$. We have the corresponding Cartan decomposition $\mk{g} = \mk{k} \oplus \mk{p}$ and the Cartan involution $\uptheta \in \Aut(\mk{g})$ that has $\mk{k}$ and $\mk{p}$ as its $(+1)$- and $(-1)$-eigenspaces, respectively. We identify $\mk{p} \cong T_oM$ in the usual way. If we denote the Killing form of $\mk{g}$ by $B$, then $B_\uptheta(X,Y) = -B(X,\uptheta Y)$ is a $\widetilde{K}$-invariant inner product on $\mk{g}$. The form $B_\uptheta$ coincides with $B$ on $\mk{p}$ and equals $-B$ on $\mk{k}$. The Cartan decomposition is orthogonal with respect to both $B$ and $B_\uptheta$. We stick with $B_\uptheta$ as our default inner product on $\mk{g}$. Whenever $U \subset V$ are subspaces of $\mk{g}$, we write $V \ominus U$ for the orthogonal complement of $U$ in $V$. We make a standard but pivotal assumption that the Riemannian metric of $M$ at $o$ coincides with the restriction of $B_\uptheta$ (hence of $B$) to $\mk{p} \cong T_oM$. We will sometimes refer to this for brevity as ``the Riemannian metric on $M$ comes from the Killing form''.

\vspace{-0.2em}

Pick a maximal abelian subspace $\mk{a}$ in $\mk{p}$ and write $r = \dim(\mk{a}) = \rank(M)$. The connected Lie subgroup $A \subset G$ with Lie algebra $\mk{a}$ is closed and abelian and its orbit $A \ccdot o \simeq \mb{E}^r$ is a maximal flat in $M$. The restriction of $B_\uptheta$ to $\mk{a}$ is an inner product and thus gives an isomorphism $\mk{a} \isoto \mk{a}^*$ and the corresponding inner product on $\mk{a}^*$. We have the restricted root system $\Upsigma \subset \mk{a}^*$, which gives rise to the restricted root space decomposition $\mk{g} = \mk{g}_0 \oplus \bigoplus_{\upalpha \in \Upsigma} \mk{g}_\upalpha$. Here $\mk{g}_0 = \mk{k}_0 \oplus \mk{a}$, where $\mk{k}_0 = Z_\mk{k}(\mk{a}) = N_\mk{k}(\mk{a})$. The summands of the restricted root space decomposition are pairwise orthogonal with respect to $B_\uptheta$. Note that $\uptheta \mk{g}_\upalpha = \mk{g}_{-\upalpha}$. Finally, we make a choice of positive roots $\Upsigma^+ \subset \Upsigma$ and write $\Uplambda = \set{\upalpha_1, \ldots, \upalpha_r} \subseteq \Upsigma^+$ for the set of simple roots. The sum $\mk{n} = \bigoplus_{\upalpha \in \Upsigma^+} \mk{g}_\upalpha$ is a nilpotent subalgebra of $\mk{g}$. We have a vector space decomposition $\mk{g} = \mk{k} \oplus \mk{a} \oplus \mk{n}$ called the Iwasawa decomposition. The connected Lie subgroup of $G$ corresponding to $\mk{n}$ is closed and will be denoted by $N$. The multiplication induces a diffeomorphism $KAN \isoto G$ called the global Iwasawa decomposition. Since $K$ is a maximal compact subgroup of $G$, the subgroups $A$ and $N$ are simply connected. Their product $S = AN$ is a simply connected closed solvable subgroup of $G$ (with Lie algebra $\mk{s} = \mk{a} \oplus \mk{n}$) that acts simply transitively on $M$. In particular, $M$ can be realized, albeit not uniquely, as a simply connected solvable Lie group with a suitable left-invariant Riemannian metric. Every root $\upalpha$ can be carried to a vector $H_\upalpha \in \mk{a}$ along the isomorphism between $\mk{a}$ and $\mk{a}^*$. On the other hand, the basis $\upalpha_1, \ldots, \upalpha_r$ for $\mk{a}^*$ gives the dual basis for $\mk{a}$, which we denote by $H^{\upalpha_1}, \ldots, H^{\upalpha_r}$. By definition, $\cross{H^{\upalpha_i}}{H_{\upalpha_j}} = \upalpha_j(H^{\upalpha_i}) = \updelta_{ij}$.

\vspace{-0.15em}

A linear endomorphism $T \in \End(\mk{a}^*)$ is called an automorphism of $\Upsigma$ if it maps $\Upsigma$ onto itself and preserves the integers $n_{\upalpha \upbeta} = \frac{2 \cross{\upalpha}{\upbeta}}{||\upbeta||^2}$, i.e. $n_{F(\upalpha)F(\upbeta)} = n_{\upalpha\upbeta}$ for all $\upalpha, \upbeta \in \Upsigma$. Such $T$ is necessarily a linear automorphism and, since the inner product on $\mk{a}^*$ comes from the Killing form $B$, in fact an orthogonal transformation of $\mk{a}^*$. The group $\Aut(\Upsigma)$ is a finite subgroup of $\O(\mk{a}^*)$ and it contains the Weyl group $\mr{W}(\Upsigma)$ as a normal subgroup. A choice of positive roots allows to split the short exact sequence $\mr{W}(\Upsigma) \hookrightarrow \Aut(\Upsigma) \twoheadrightarrow \Aut(\Upsigma)/\mr{W}(\Upsigma)$ as a semidirect product. Indeed, having fixed $\Upsigma^+$, note that $\Uplambda$ gives rise to a Dynkin diagram, which we denote by $\DD_M$. We treat it as a vertex-weighted graph (with some edges oriented and doubled/tripled and some vertices "marked" in case they correspond to a simple root whose double is also root) where each vertex is assigned the multiplicity of the corresponding simple root, i.e. the dimension of the corresponding root subspace. We will be working with automorphisms of $\DD_M$, which we require by definition to preserve the multiplicities of the vertices. Any such automorphism $P \in \Aut(\DD_M)$ gives rise to a linear automorphism of $\mk{a}^*$ -- denoted by the same letter -- simply by permuting the basis vectors in $\Uplambda$. One can easily see that $P \in \Aut(\Upsigma)$ and the map $\Aut(\DD_M) \to \Aut(\Upsigma)$ is an injective group homomorphism. Moreover, its image and $\mr{W}(\Upsigma)$ do not intersect and their product is the whole automorphism group of $\Upsigma$, hence $\Aut(\Upsigma) \simeq \mr{W}(\Upsigma) \rtimes \Aut(\DD_M)$. Note that there may be elements in $\Aut(\DD_M)$ interchanging some connected components of $\DD_M$ -- precisely when $M$ has isometric de Rham factors or, equivalently, when $\mk{g}$ has isomorphic simple ideals (this is one of the several places where the assumption on the Riemannian metric of $M$ plays an important role).

\vspace{-0.57em}

Now, we have a compact group $N_{\widetilde{K}}(\mk{a})$ and its representation $\uppsi \colon N_{\widetilde{K}}(\mk{a}) \to \O(\mk{a}^*), k \mapsto (\restr{\Ad(k)}{\mk{a}}^*)^{-1}$. One readily sees that $\Ker(\uppsi)= Z_{\widetilde{K}}(\mk{a})$ and $\Im(\uppsi) \subseteq \Aut(\Upsigma)$, which implies that $\Ad(k)$ permutes the restricted root subspaces: $\Ad(k) \mk{g}_\upalpha = \mk{g}_{\uppsi(k)\upalpha}$. Since the Riemannian metric on $M$ comes from the Killing form $B$, one can show that the image of $\uppsi$ equals the whole automorphism group of $\Upsigma$, hence $N_{\widetilde{K}}(\mk{a})/Z_{\widetilde{K}}(\mk{a}) \cong \Aut(\Upsigma)$. We can also give a similar description of the subgroups $\mr{W}(\Upsigma)$ and $\Aut(\DD_M)$ of $\Aut(\Upsigma)$. Namely, consider the normalizers $N_K(\mk{a})$ and $N_{\widetilde{K}}(\mk{n})$. These are both subgroups of $N_{\widetilde{K}}(\mk{a})$, and we have $N_K(\mk{a}) \cap Z_{\widetilde{K}}(\mk{a}) = Z_K(\mk{a})$ and $N_{\widetilde{K}}(\mk{n}) \supseteq Z_{\widetilde{K}}(\mk{a})$. It turns out that $\uppsi(N_K(\mk{a})) = \mr{W}(\Upsigma)$ and $\uppsi(N_{\widetilde{K}}(\mk{n})) = \Aut(\DD_M)$, so $N_K(\mk{a})/Z_K(\mk{a}) \cong \mr{W}(\Upsigma)$ and $N_{\widetilde{K}}(\mk{n})/Z_{\widetilde{K}}(\mk{a}) \cong \Aut(\DD_M)$. Note that all of these subgroups of $\widetilde{K}$ share the same Lie algebra $\mk{k}_0$.

\vspace{-0.53em}

In order to be able to describe the results of \cite{hyperpolarfoliations}, we will need some basic facts from the theory of parabolic subgroups, which we now briefly introduce. Let $\Upphi$ be any subset of  $\Uplambda$. Write $\Upsigma_\Upphi$ for the root subsystem of $\Upsigma$ spanned by $\Upphi$ and let $\Upsigma_\Upphi^+ = \Upsigma_\Upphi \cap \Upsigma^+$ be the induced choice of positive roots in $\Upsigma_\Upphi$. We start by dividing $\mk{a}$ into the orthogonal sum of two subspaces: $\mk{a}_\Upphi = \bigcap_{\upalpha \in \Upphi} \Ker (\upalpha) = \bigoplus_{\upalpha \in \Uplambda \mysetminus \Upphi} \Rl H^\upalpha$, $\mk{a}^\Upphi = \mk{a} \ominus \mk{a}_\Upphi = \bigoplus_{\upalpha \in \Upphi} \Rl H_\upalpha$. Next, we introduce $\mk{l}_\Upphi = N_\mk{g}(\mk{a}_\Upphi) = Z_\mk{g}(\mk{a}_\Upphi) = \mk{g}_0 \oplus \bigoplus_{\upalpha \in \Upsigma_\Upphi} \mk{g}_\upalpha$ and $\mk{n}_\Upphi = \bigoplus_{\upalpha \in \Upsigma^+ \mysetminus \Upsigma_\Upphi^+} \mk{g}_\upalpha \subseteq \mk{n}$, which are reductive and nilpotent Lie subalgebras of $\mk{g}$, respectively. One has $\mk{l}_\Upphi \cap \mk{n}_\Upphi = \set{0}$ and $[\mk{l}_\Upphi, \mk{n}_\Upphi] \subseteq \mk{n}_\Upphi$. The semidirect sum $\mk{q}_\Upphi = \mk{l}_\Upphi \loplus \mk{n}_\Upphi$ is a parabolic subalgebra of $\mk{g}$. Of particular interest to us will be the reductive subalgebra $\mk{m}_\Upphi = \mk{l}_\Upphi \ominus \mk{a}_\Upphi$ and its semisimple subalgebra $\widetilde{\mk{g}}_\Upphi = \mk{m}_\Upphi \ominus Z_{\mk{k}_0}(\mk{b}_\Upphi)$, where $\mk{b}_\Upphi = \mk{m}_\Upphi \cap \mk{p}$. One can show that $\widetilde{\mk{g}}_\Upphi$ is generated by $\mk{b}_\Upphi$ or $\bigoplus_{\upalpha \in \Upsigma_\Upphi} \mk{g}_\upalpha$ as a Lie algebra (see Subsection 2.4 and Remark 2.4 in \cite{mypaper} for more details on $\widetilde{\mk{g}}_\Upphi$). The decomposition $\mk{q}_\Upphi = \mk{m}_\Upphi \oplus \mk{a}_\Upphi \loplus \mk{n}_\Upphi$ is called the Langlands decomposition. Finally, define $\mk{k}_\Upphi = \mk{k} \cap \mk{m}_\Upphi$.

\vspace{-0.3em}

Now we look at the subgroups of $G$ and submanifolds of $M$ arising from these Lie subalgebras. Let $\widetilde{G}_\Upphi, A_\Upphi,$ and $N_\Upphi$ stand for the connected Lie subgroups of $G$ corresponding to $\widetilde{\mk{g}}_\Upphi, \mk{a}_\Upphi,$ and $\mk{n}_\Upphi$, respectively. They are all closed and thus determine properly embedded submanifolds $\widetilde{G}_\Upphi \ccdot o = B_\Upphi, A_\Upphi \ccdot o,$ and $N_\Upphi \ccdot o$, the first two of which are totally geodesic with the corresponding Lie triple systems $\mk{b}_\Upphi$ and $\mk{a}_\Upphi$. The submanifold $B_\Upphi$ is called a \textit{boundary component of} $M$ (the term coined in \cite{borel_ji} in the context of the maximal Satake compactification of $M$), and it is itself a noncompact symmetric space of rank equal to $r_\Upphi = |\Upphi|$. The subgroup $L_\Upphi = Z_G(\mk{a}_\Upphi)$ is closed and reductive and has $\mk{l}_\Upphi$ as its Lie algebra. The group $K_\Upphi = K \cap L_\Upphi$ is a maximal compact subgroup of $L_\Upphi$ and has Lie algebra $\mk{k}_\Upphi$. The product $Q_\Upphi = L_\Upphi N_\Upphi$ is a parabolic subgroup of $G$; it is closed, acts transitively on $M$, and has $\Lie(Q_\Upphi) = \mk{q}_\Upphi$. Define also $M_\Upphi = K_\Upphi \widetilde{G}_\Upphi \subseteq L_\Upphi$, which is a closed reductive subgroup of $G$ with $\Lie(M_\Upphi) = \mk{m}_\Upphi$. One has a direct product decomposition $L_\Upphi = M_\Upphi \times A_\Upphi$, which induces the global Langlands decomposition $Q_\Upphi = M_\Upphi \times A_\Upphi \ltimes N_\Upphi$. The latter allows us to draw a commutative diagram: \\
$$
\xymatrix{
M_\Upphi \times A_\Upphi \times N_\Upphi \ar[r]^-{\sim} \ar@{->>}[d] & Q_\Upphi \ar@{->>}[d] \ar@{->>}[dr] \\
B_\Upphi \times A_\Upphi \times N_\Upphi \ar[r]^-{\sim} & Q_\Upphi/K_\Upphi \ar[r]^(0.57){\sim} & M .
}
$$
The diffeomorphism $B_\Upphi \times A_\Upphi \times N_\Upphi \simeq M$ is called a \textit{horospherical decomposition of} $M$.

Of special interest to us will be subsets $\Upphi \subseteq \Uplambda$ such that no two roots in $\Upphi$ are connected by an edge in the Dynkin diagram $\DD_M$. In other words, all roots in $\Upphi$ should be mutually orthogonal. Such subsets are called \textit{orthogonal} in \cite{hyperpolarfoliations}. Fix one such subset $\Upphi \subseteq \Uplambda$. For any $\upalpha, \upbeta \in \Upphi, \upalpha \ne \upbeta,$ and any $k,l \ne 0$, we have $\mk{g}_{k\upalpha + l\upbeta} = \set{0}$ and hence $[\mk{g}_{k\upalpha}, \mk{g}_{l\upbeta}] = \set{0}$. Since $\widetilde{\mk{g}}_\Upphi$ is generated by $\bigoplus_{\upalpha \in \Upsigma_\Upphi} \mk{g}_\upalpha$, we have a direct sum Lie algebra decomposition
$$
\widetilde{\mk{g}}_\Upphi = \bigoplus_{\upalpha \in \Upphi} \widetilde{\mk{g}}_{\{\upalpha \}}.
$$
Consequently, the multiplication map $\prod_{\upalpha \in \Upphi} \widetilde{G}_{\{\upalpha\}} \twoheadrightarrow \widetilde{G}_\Upphi$ is a local isomorphism, and it passes to a Riemannian covering $\prod_{\upalpha \in \Upphi} B_{\{\upalpha\}} \twoheadrightarrow B_\Upphi$. Since $B_\Upphi$ is simply connected, we get an isometric decomposition
$$
B_\Upphi \cong \prod_{\upalpha \in \Upphi} B_{\{\upalpha\}}.
$$
Note that each $B_{\{\upalpha\}}$ is a noncompact symmetric space of rank 1 and thus is isometric to a hyperbolic space over a finite-dimensional real normed division algebra.

\subsection{Isometries of a Riemannian product}

In order to tell foliations $\mc{F}_{\upalpha_i}$ apart, we will need a certain structural result about the isometry group of a Riemannian manifold decomposed as a Riemannian product. Although we will only work with the case of the de Rham decomposition of a symmetric space of noncompact type, we formulate and prove this result in a much more general setting. In the following, whenever we say that a connected Riemannian manifold is irreducible, we mean that it is not flat and its restricted holonomy representation is irreducible. If the manifold is complete, this is equivalent to requiring its universal Riemannian covering to be (non-flat and) indecomposable as a nontrivial Riemannian product.

Let $M = M_0 \times M_1^{l_1} \times \cdots \times M_k^{l_k}$ be a Riemannian product, where all $M_i$'s are connected Riemannian manifolds, $M_0$ is flat, $M_i$ is irreducible for $1 \leqslant i \leqslant k$, and $M_i^{l_i}$ simply means the product of $l_i$ copies of $M_i$. We also assume that $M_i$ is not isometric to $M_j$ for $i \ne j$. Note that we have a Lie subgroup\footnote{This is a Lie subgroup because the induced isometric action $I(M_0) \times I(M_1)^{l_1} \times \cdots \times I(M_k)^{l_k} \curvearrowright M$ is smooth and effective.} $I(M_0) \times I(M_1)^{l_1} \times \cdots \times I(M_k)^{l_k} \subseteq I(M)$. Let $l = \sum_{i=1}^k l_i$, write $S_l$ for the symmetric group on $l$ elements, and let $S_l^{l_1, \ldots, l_k} \cong S_{l_1} \times \cdots \times S_{l_k}$ stand for the subgroup of elements that permute the first $l_1$ elements with each other, the next $l_2$ elements with each other, and so on. Observe that we have an embedding $S_l^{l_1, \ldots, l_k} \hookrightarrow I(M)$ given by the rule $\upsigma \cdot (p_0, (p_s)) = (p_0, (p_{\upsigma(s)}))$ (to be precise, this map is an antihomomorphism of groups).

\begin{proposition}\label{product_isometries}
The group $I(M)$ decomposes as a semidirect product of its subgroups
$$
\left[ I(M_0) \times I(M_1)^{l_1} \times \cdots \times I(M_k)^{l_k} \right] \rtimes S_l^{l_1, \ldots, l_k} = I(M).
$$
In particular, $I(M_0) \times I(M_1)^{l_1} \times \cdots \times I(M_k)^{l_k}$ is an open normal subgroup of $I(M)$. The corresponding action of $S_l^{l_1, \ldots, l_k}$ on it is given by $\upsigma \cdot (g_0, (g_s)) = (g_0, (g_{\upsigma(s)}))$.
\end{proposition}

One obvious example of such a decomposition is the de Rham decomposition of a complete simply connected Riemannian manifold. Note that we do not require the factors to be complete in the proposition.

\begin{proof}
The subgroups $I(M_0) \times I(M_1)^{l_1} \times \cdots \times I(M_k)^{l_k}$ and $S_l^{l_1, \ldots, l_k}$ of $I(M)$ clearly do not intersect. So we only need to show that their product is the whole isometry group. Let $p = (p_0, (p_1^s)_{s=1}^{l_1}, \ldots, (p_k^s)_{s=1}^{l_k}) \in M$ and consider the decomposition 
$$
T_pM = T_{p_0}M_0 \oplus \left( \bigoplus_{s=1}^{l_1} T_{p_1^s} M_1 \right) \oplus \cdots \oplus \left( \bigoplus_{s=1}^{l_k} T_{p_k^s} M_k \right).
$$
Let $\mathrm{Hol}^0(M,p)$ stand for the restricted holonomy group of $M$ at $p$, which is defined using the parallel transport along all contractible piecewise-smooth loops based at $p$ and is the identity component of the full holonomy group $\mathrm{Hol}(M,p)$. The decomposition of $T_pM$ above is obviously orthogonal, the first summand $T_{p_0}M_0$ is the subspace of invariants of $\mathrm{Hol}^0(M,p)$ in $T_pM$, and all the other summands are irreducible $\mathrm{Hol}^0(M,p)$-subrepresentations in $T_pM$. A decomposition of $T_pM$ satisfying these three properties is called a \textit{canonical decomposition} in \cite{kobayashi_nomizu_I}. Now, Theorem 5.4(4) in Chapter IV of the book asserts that a canonical decomposition of $T_pM$ is unique up to reordering of its factors. It is worth noting that the authors actually talk about decompositions with respect to the full holonomy group, and Theorem 5.4(4) requires $M$ to be simply connected. But note that the restricted holonomy representation of $M$ is isomorphic to the full holonomy representation of its universal Riemannian covering $\widetilde{M}$ (simply by lifting contractible loops). So the uniqueness of a canonical $\mathrm{Hol}(\widetilde{M}, \widetilde{p})$-decomposition for $\widetilde{M}$ (here $\widetilde{p} \in \widetilde{M}$ is any point over $p$) translates into the uniqueness of a canonical $\mathrm{Hol}^0(M, p)$-decomposition for $M$ (see \cite[Ch. IV, Section 5]{kobayashi_nomizu_I} for more details). Let $g \in I(M)$ be any isometry. It is not hard to show that the differential $dg$ must send canonical decompositions of $T_pM$ to canonical decompositions of $T_{g(p)}M$ (basically because isometries commute with parallel transport). Let us write $M_{0,p}$ for $M_0 \times \{((p_1^s)_{s=1}^{l_1}, \ldots, (p_k^s)_{s=1}^{l_k})\}$ and $M_{i,p}^{(j)}$ for $\{(p_0, (p_1^s)_{s=1}^{l_1}, \ldots, (p_i^s)_{s=1}^{j-1}))\} \times M_i \times \{((p_i^s)_{s=j+1}^{l_i}, \ldots, (p_k^s)_{s=1}^{l_k})\}$ for any $i \in \set{1, \ldots, k}$ and $j \in \set{1, \ldots, l_i}$. These are totally geodesic submanifolds of $M$. Since isometries commute with the exponential map and respect canonical decompositions, $g$ must send $M_{0,p}$ onto $M_{0, g(p)}$ and $M_{i,p}^{(j)}$ onto $M_{i,g(p)}^{(j')}$ for some $j' \in \set{1, \ldots, l_i}$. If we write $j = \upsigma(j')$, we obtain a permutation $\upsigma \in S_{l_1} \times \cdots \times S_{l_k} \cong S_l^{l_1, \ldots,l_k}$. We also have isometries $M_0 \cong M_{0,p} \xrightarrow{g} M_{0, g(p)} \cong M_0$ and $M_i \cong M_{i,p}^{(j)} \xrightarrow{g} M_{i,g(p)}^{(j')} \cong M_i$, which we denote by $g_0$ and $g_i^{(j')}$, respectively. By construction, the isometry $(g_0, (g_1^{(s)})_{s=1}^{l_1}, \ldots, (g_k^{(s)})_{s=1}^{l_k}) \circ \upsigma$ lies in the product of $I(M_0) \times I(M_1)^{l_1} \times \cdots \times I(M_k)^{l_k}$ and $S_l^{l_1, \ldots, l_k}$ and coincides with $g$ on $M_{0,p} \cup \bigcup_{\substack{1 \leqslant i \leqslant k \\ 1 \leqslant j \leqslant l_i}} M_{i,p}^{(j)}$. But then the differentials of these two isometries at $p$ must coincide as well. Since an isometry of a connected Riemannian manifold is uniquely determined by its value at a point and its differential at that point, the constructed isometry coincides with $g$, which finishes the proof.
\end{proof}

\begin{corollary}
We have as isomorphism $I^0(M_0) \times I^0(M_1)^{l_1} \times \cdots \times I^0(M_k)^{l_k} \cong I^0(M)$.
\end{corollary}

\vspace{0.5em}

\section{Main result}\label{main_section}

In this section we will discuss the classification results of Berndt, Tamaru, and D\'{i}az-Ramos and prove the Main Theorem. Historically, their paper \cite{hyperpolarfoliations} about homogeneous hyperpolar foliations came out after Berndt and Tamaru obtained the classification of homogeneous codimension-one foliations in \cite{berndttamarufoliations}. We will, however, present the results of these two papers in an anachronistic manner, as it makes the exposition clearer.

In \cite{hyperpolarfoliations}, Berndt, Tamaru, and D\'{i}az-Ramos devised a method allowing one to obtain homogeneous hyperpolar foliations on an arbitrary symmetric space of noncompact type from those on Euclidean and hyperbolic spaces and then showed that the foliations obtained in this way exhaust the list of all homogeneous hyperpolar foliations up to congruence. We begin by introducing their method.

Let $M$ be a symmetric space of noncompact type, and let the rest of the notation be as in Subsection \ref{IHdecompositions}. Let $\Upphi \subseteq \Uplambda$ be an orthogonal subset. Recall that we have a horospherical decomposition
\begin{equation}\label{hordecomporthog}
M \simeq B_\Upphi \times A_\Upphi \times N_\Upphi \simeq \left( \prod_{\upalpha \in \Upphi} B_{\{\upalpha\}} \right) \times A_\Upphi \times N_\Upphi.
\end{equation}
Now, each boundary component $B_{\{\upalpha\}}$ is a symmetric space of noncompact type and rank 1, hence it is isometric to a hyperbolic space $\mathbb{F}_\upalpha H^{n_\upalpha}, \, \mathbb{F}_\upalpha \in \set{\Rl, \Cx, \Hq, \Oo}$. Every nontrivial homogeneous hyperpolar foliation on this space is of codimension one because its sections are totally geodesic and they must be flat by definition, hence their dimension cannot exceed the rank of the space. As $\mathbb{F}_\upalpha H^{n_\upalpha}$ is irreducible, the classification of homogeneous codimension-one foliations obtained in \cite{berndttamarufoliations} (to be formulated below, see Theorem \ref{berndttamarufoliations}) tells us that there are exactly two such foliations on this space up to congruence: one of them is a foliation by horospheres all congruent to each other, while the other, which we denote\footnote{Throughout the paper we commit the usual sin of not distinguishing notationally between a congruence class of foliations and its specific representatives. For example, $\mathcal{F}_{\mathbb{F}_\upalpha}^{n_\upalpha}$ is really a congruence class of homogeneous codimension-one foliations on $\mathbb{F}_\upalpha H^{n_\upalpha}$.} by $\mathcal{F}_{\mathbb{F}_\upalpha}^{n_\upalpha}$, has a unique minimal leaf. On the other hand, the orbit $A_\Upphi \ccdot o$ is a simply connected totally geodesic flat submanifold of $M$ of dimension $r - r_\Upphi$, so it is isometric to $\mb{E}^{r-r_\Upphi}$. We can also think of it as $\mk{a}_\Upphi$ endowed with the Riemannian metric coming from its inner product or the subgroup $A_\Upphi$ endowed with the left-invariant metric corresponding to the inner product on $\mk{a}_\Upphi$. We have the isometries $\exp_{A_\Upphi} \colon \mk{a}_\Upphi \isoto A_\Upphi$ and $A_\Upphi \isoto A_\Upphi \ccdot o, \, g \mapsto g \ccdot o$. Every homogeneous hyperpolar (which is the same as polar in this context) foliation on $\mk{a}_\Upphi$ is by affine subspaces parallel to a fixed linear subspace $V \subseteq \mk{a}_\Upphi$: $\mathcal{F}_V = \set{x + V \mid x \in V^\perp \subseteq \mk{a}_\Upphi}$. We denote the corresponding foliation on $A_\Upphi$ (or $A_\Upphi \ccdot o$) by the same symbol $\mathcal{F}_V$.

Going back to the horospherical decomposition \eqref{hordecomporthog}, consider the product foliation
$$
\mathcal{F}_{\Upphi, V} = \left( \prod_{\upalpha \in \Upphi} \mathcal{F}_{\mathbb{F}_\upalpha}^{n_\upalpha} \right) \times \mathcal{F}_V \times N_\Upphi
$$
on $M$, where the last factor is simply the trivial foliation on $N_\Upphi$ consisting of just one leaf. The main result of \cite{hyperpolarfoliations} is

\begin{theorem}[\cite{hyperpolarfoliations}, Th.\,4.11]\label{hyperpolarfoliations}
Let $M$ be a symmetric space of noncompact type with all notation as above. Then

\begin{enumerate}
    \item For every orthogonal subset $\Upphi \subseteq \Uplambda$ and linear subspace $V \subseteq \mk{a}_\Upphi$, $\mathcal{F}_{\Upphi, V}$ is a homogeneous hyperpolar foliation on $M$.
    \item Every homogeneous hyperpolar foliation on $M$ is isometrically congruent to $\mathcal{F}_{\Upphi, V}$ for some orthogonal subset $\Upphi \subseteq \Uplambda$ and linear subspace $V \subseteq \mk{a}_\Upphi$.
\end{enumerate}
\end{theorem}

In order to construct a Lie subgroup of $G$ whose orbits are the leaves of $\mathcal{F}_{\Upphi, V}$, pick an $r_\Upphi$-dimensional linear subspace $\ell_\Upphi \subset \mk{g}$ such that $\dim(\mk{g}_\upalpha \cap \ell_\Upphi) = 1$ for each $\upalpha \in \Upphi$ and consider the subspace
$$
\mk{s}_{\Upphi, V} = (\mk{a}^\Upphi \oplus V) \oplus (\mk{n} \ominus \ell_\Upphi) \subseteq \mk{a} \oplus \mk{n} = \mk{s}. 
$$
One can easily check that this is a Lie subalgebra and the corresponding connected Lie subgroup $S_{\Upphi, V} \subseteq G$ is closed and acts on $M$ without singular orbits. Its action turns out to be hyperpolar and its orbit foliation is exactly $\mathcal{F}_{\Upphi, V}$. A different choice of $\ell_\Upphi$ would result in a subgroup congruent to $S_{\Upphi, V}$ and thus a foliation congruent to $\mathcal{F}_{\Upphi, V}$.

Notably, Theorem \ref{hyperpolarfoliations} allows $M$ to be reducible. On the other hand, the assumption that the Riemannian metric on $M$ comes from the Killing form of $\mk{g}$ seems to play a crucial role in the proof of this theorem in \cite{hyperpolarfoliations} and we do not know whether it can be safely removed. Note that this result gives only partial classification of homogeneous hyperpolar foliations, for we may have $\mathcal{F}_{\Upphi, V}$ congruent to $\mathcal{F}_{\Upphi', V'}$ for $(\Upphi, V) \ne (\Upphi', V')$, and we have no way of telling when this is the case in general.

Let us apply Theorem \ref{hyperpolarfoliations} to homogeneous codimension-one foliations. Observe that $\codim(\mathcal{F}_{\Upphi, V}) = |\Upphi| + \codim_{\mk{a}_\Upphi}(V)$. Therefore, there are two types of such foliations:

\begin{enumerate}
    \item $\mathcal{F}_\ell = \mathcal{F}_{\Upphi, V}$ with $\Upphi = \varnothing$ and $V = \mk{a} \ominus \ell$, where $\ell \subseteq \mk{a}$ is a one-dimensional linear subspace, and
    \item $\mathcal{F}_{\upalpha_i} = \mathcal{F}_{\Upphi, V}$ with\footnote{This foliation is denoted by $\mathcal{F}_i$ in \cite{berndttamarufoliations} but our notation will prove less ambiguous in the reducible case so we stick with it.} $\Upphi = \set{\upalpha_i}$ and $V = \mk{a}_\Upphi$.
\end{enumerate}

The subalgebra $\mk{s}_{\Upphi, V}$ (resp., subgroup $S_{\Upphi, V}$) in this case will be denoted simply by $\mk{s}_\ell$ (resp., $S_\ell$) or $\mk{s}_{\upalpha_i}$ (resp., $S_{\upalpha_i}$). Note that $\mk{s}_\ell = (\mk{a} \ominus \ell) \oplus \mk{n}$ and $\mk{s}_{\upalpha_i} = \mk{a} \oplus (\mk{n} \ominus \ell_i)$, where $\ell_i \subseteq \mk{g}_{\upalpha_i}$ is a line. Berndt and Tamaru showed in \cite{berndttamarufoliations} that $\mathcal{F}_\ell$ has all its leaves isometrically congruent to each other, while $\mathcal{F}_{\upalpha_i}$ has a unique minimal leaf, namely $S_{\upalpha_i} \ccdot o$. In particular, $\mathcal{F}_\ell$ is never congruent to $\mathcal{F}_{\upalpha_i}$. It means that in order to complete the classification of homogeneous codimension-one foliations, one is left to solve the congruence problem on the parameter space $\mathbb{P}\mk{a} \sqcup \set{\upalpha_1, \ldots, \upalpha_r}$, that is, understand when $\mathcal{F}_\ell$ is congruent to $\mathcal{F}_{\ell'}$ and $\mathcal{F}_{\upalpha_i}$ to $\mathcal{F}_{\upalpha_j}$. This was done in \cite{berndttamarufoliations} in the irreducible case. Before formulating this result, let us note that every $P \in \Aut(\DD_M) \subseteq \Aut(\Upsigma)$ acts orthogonally on $\mk{a}$ as $\widehat{P} \defeq (P^*)^{-1}$. The operator $\widehat{P}$ can be explicitly described in terms of the basis $H_{\upalpha_1}, \ldots, H_{\upalpha_r}$ of $\mk{a}$ by sending $H_{\upalpha_i}$ to $H_{P(\upalpha_i)}$, or, equivalently, in terms of the basis $H^{\upalpha_1}, \ldots, H^{\upalpha_r}$ by sending $H^{\upalpha_i}$ to $H^{P(\upalpha_i)}$. If $k \in N_{\widetilde{K}}(\mk{a})$ is such that $\uppsi(k) = P$ (such $k$ must lie in $N_{\widetilde{K}}(\mk{n}) \subseteq N_{\widetilde{K}}(\mk{a})$), then $\widehat{P}$ coincides with $\restr{\Ad(k)}{\mk{a}}$. Altogether, we have the actions of $\Aut(\DD_M)$ on $\mk{a}$ and $\Uplambda$.

\begin{proposition}\label{easypart}
If $\ell, \ell' \in \mathbb{P}\mk{a}$ (respectively, $\upalpha_i, \upalpha_j \in \Uplambda$) lie in the same $\Aut(\mathrm{DD_M})$-orbit, then the corresponding foliations $\mathcal{F}_\ell$ and $\mathcal{F}_{\ell'}$ (respectively, $\mathcal{F}_{\upalpha_i}$ and $\mathcal{F}_{\upalpha_j}$) are isometrically congruent.
\end{proposition}
\vspace{-0.5em}
\begin{proof}
Given $P \in \Aut(\DD_M)$ sending $\ell$ onto $\ell'$, take any $k \in N_{\widetilde{K}}(\mk{n})$ such that $\uppsi(k) = P$. By design, $\Ad(k)$ preserves both $\mk{a}$ and $\mk{n}$ and thus restricts to an automorphism of $\mk{s}$ that sends $\ell$ onto $\ell'$ and thus $\mk{s}_\ell = \mk{s} \ominus \ell$ onto $\mk{s}_{\ell'} = \mk{s} \ominus \ell'$. Consequently, $k$ is the desired congruence between $\mathcal{F}_\ell$ and $\mathcal{F}_{\ell'}$. The argument for $\mathcal{F}_{\upalpha_i}$ and $\mathcal{F}_{\upalpha_j}$ is analogous.
\end{proof}
\vspace{-0.5em}
The converse is true provided that $M$ is irreducible:

\begin{theorem}[\cite{berndttamarufoliations}, Th.\,3.5 and 4.8]\label{berndttamarufoliations}
Let $M$ be an irreducible symmetric space of noncompact type, and assume that the foliations $\mathcal{F}_\ell$ and $\mathcal{F}_{\ell'}$ (respectively, $\mathcal{F}_{\upalpha_i}$ and $\mathcal{F}_{\upalpha_j}$) are isometrically congruent. Then there exists $P \in \Aut(\mathrm{DD_M})$ mapping $\ell$ onto $\ell'$ (respectively, $\upalpha_i$ to $\upalpha_j$). Consequently, the moduli space of homogeneous codimension-one foliations on $M$ is isomorphic to
$$
(\mathbb{R}P^{r-1} \sqcup \set{1, \ldots, r})/\Aut(\mathrm{DD_M}).
$$
\end{theorem}
\vspace{-0.5em}
We will extend this result verbatim to the general (reducible) case. First, we deal with the discrete part of the moduli space.

\begin{proposition}\label{congruenceFi}
Let $M$ be a symmetric space of noncompact type. Assume that the foliations $\mathcal{F}_{\upalpha_i}$ and $\mathcal{F}_{\upalpha_j}$ are isometrically congruent. Then there exists $P \in \Aut(\mathrm{DD_M})$ mapping $\upalpha_i$ to $\upalpha_j$.
\end{proposition}

\begin{proof}
Let $M = M_1 \times \cdots \times M_k$ be the de Rham decomposition of $M$. We write $G_i = I^0(M_i)$ and $\mk{g}_i = \Lie(G_i)$, so we have $G = G_1 \times \cdots \times G_k$ and $\mk{g} = \mk{g}_1 \oplus \cdots \oplus \mk{g}_k$. The Cartan involution on $\mk{g}$ preserves its decomposition into simple ideals so we can write $\mk{g}_i = \mk{k}_i \oplus \mk{p}_i$. We also have $\mk{a} = \mk{a}_1 \oplus \cdots \oplus \mk{a}_k$ and $\Upsigma^+ = \Upsigma_1^+ \sqcup \cdots \sqcup \Upsigma_k^+$, where $\mk{a}_i = \mk{a} \cap \mk{g}_i$ and $\Upsigma_i^+ = \Upsigma^+ \cap \Upsigma_i$. This induces the corresponding decompositions $\mk{s} = \mk{s}_1 \oplus \cdots \oplus \mk{s}_k, \, S = S_1 \times \cdots \times S_k,$ and $\Uplambda = \Uplambda_1 \sqcup \cdots \sqcup \Uplambda_k$, where we denote $\Uplambda_l = \set{\upalpha_l^1, \ldots, \upalpha_l^{r_l}}$. To comply with this notation, we write $\upalpha_i^p$ and $\upalpha_j^q$ instead of $\upalpha_i$ and $\upalpha_j$. Pick some lines $\ell_i^p \subseteq \mk{g}_{\upalpha_i^p}$ and $\ell_j^q \subseteq \mk{g}_{\upalpha_j^q}$. We have the orbit foliations of $S_{\upalpha_i^p}$ and $S_{\upalpha_j^q}$ as representatives of $\mathcal{F}_{\upalpha_i^p}$ and $\mathcal{F}_{\upalpha_j^q}$. Observe that 
\begin{align}
S_{\upalpha_i^p} &= S_1 \times \cdots \times S_{i, \upalpha_i^p} \times \cdots \times S_k, \; \text{so} \nonumber \\
S_{\upalpha_i^p} \ccdot o &= M_1 \times \cdots \times (S_{i, \upalpha_i^p} \ccdot o) \times \cdots \times M_k, \label{orbit_decomposition}
\end{align}
where $S_{i, \upalpha_i^p}$ is the connected Lie subgroup of $G_i$ with Lie algebra $\mk{s}_{i, \upalpha_i^p} = \mk{s}_i \ominus \ell_i^p$. The same is true for $S_{\upalpha_j^q}$ and its orbit $S_{\upalpha_j^q} \ccdot o$. Let $g \in \widetilde{G}$ be a congruence identifying the orbit foliations of $S_{\upalpha_i^p}$ and $S_{\upalpha_j^q}$ and thus their minimal leaves $S_{\upalpha_i^p} \ccdot o$ and $S_{\upalpha_j^q} \ccdot o$. Without loss of generality, we may assume that $g$ fixes $o$. Looking at \eqref{orbit_decomposition} and Proposition \ref{product_isometries}, we see that $g$ must send $M_{i,o} = G_i \ccdot o$ onto $M_{j,o} = G_j \ccdot o$ and thus provide a congruence between the orbit foliations of $S_{i, \upalpha_i^p}$ on $M_{i,o} \cong M_i$ and $S_{j, \upalpha_j^q}$ on $M_{j,o} \cong M_j$ (in particular, it follows that $M_i$ and $M_j$ are isometric). By Theorem \ref{berndttamarufoliations}, there exists an isomorphism $P$ between the Dynkin diagrams $\DD_{M_i}$ and $\DD_{M_j}$ sending $\upalpha_i^p$ to $\upalpha_j^q$. Note that these two Dynkin diagrams can be regarded as connected components of $\DD_M$. Therefore, we can extend $P$ to $\DD_M$ by letting it be $P^{-1}$ on $\DD_{M_j}$ and the identity on all the components other than $\DD_{M_i}$ and $\DD_{M_j}$. This gives an element of $\Aut(\DD_M)$ sending $\upalpha_i^p$ to $\upalpha_j^q$, as required.
\end{proof}

One would expect the congruence problem for $\mathcal{F}_\ell$'s to be subtler in the reducible case, for $\ell$ does not have to be positioned nicely with respect to the de Rham decomposition. In other words, if we write $\mk{a} = \mk{a}_1 \oplus \cdots \oplus \mk{a}_k$ as we did in the proof of Proposition \ref{congruenceFi}, then $\ell$ does not have to be contained in any of the summands. Nevertheless, this problem has recently been solved in greater generality for all reducible spaces in a way that does not use Berndt and Tamaru's classification in the irreducible case. Namely, in their recent paper \cite{DV-SL_isomparametric_hypersurfaces}, Dom\'{i}nguez-V\'{a}zquez and Sanmart\'{i}n-L\'{o}pez studied certain isoparametric families of hypersurfaces in symmetric spaces of noncompact type constructed in a way very similar to that for $\mc{F}_\ell$. Instead of taking Lie subalgebras of the form $(\mk{a} \ominus \ell) \oplus \mk{n}$, they consider those of the form $\mk{b} \oplus \mk{n}$, where $\mk{b}$ is a subspace of $\mk{a}$ satisfying certain properties. The following result solves the congruence problem for such families:

\begin{theorem}[\cite{DV-SL_isomparametric_hypersurfaces}, Th.\,B(iii)]\label{DV-SM}
Let $M$ be a symmetric space of noncompact type, let $\mk{b}, \mk{b}' \subseteq \mk{a}$ be two linear subspaces, and let $BN$ and $B'N$ be the subgroups of $S$ corresponding to the Lie subalgebras $\mk{b} \oplus \mk{n}$ and $\mk{b}' \oplus \mk{n}$ of $\mk{s}$, respectively. Then the families of equidistant tubes (or equidistant hypersurfaces in case $\dim (\mk{b}) = \dim (\mk{b}') = r-1$) around the orbits $BN \ccdot o$ and $B'N \ccdot o$ are congruent to each other if and only if there exists $P \in \Aut(\DD_M)$ whose corresponding orthogonal transformation $\widehat{P}$ of $\mk{a}$ sends $\mk{b}$ onto $\mk{b}'$.
\end{theorem}

The way this result is formulated in the paper is slightly different. Instead of $\widehat{P}$, the authors talk about $\restr{\Ad(k)}{\mk{a}}$, where $k$ is an element of $N_{\widetilde{K}}(\mk{a})$ preserving $\set{H_\upalpha \mid \upalpha \in \Upsigma^+}$. But note that this is the same as to ask that $k$ lies in $N_{\widetilde{K}}(\mk{n})$, and we know that the image of the latter in $\O(\mk{a}^*)$ is precisely $\Aut(\DD_M)$. Note that the proof of part (iii) of Theorem B in \cite{DV-SL_isomparametric_hypersurfaces} works without the restrictions imposed on $\mk{b}$ and $\mk{b}'$ at the beginning of the theorem.

When we have one-dimensional subspaces $\ell, \ell' \subseteq \mk{a}$, Theorem \ref{DV-SM} applied to $\mk{a}_\ell = \mk{a} \ominus \ell$ and $\mk{a}_{\ell'} = \mk{a} \ominus \ell'$ yields:

\vspace{-0.1em}

\begin{proposition}\label{congruenceFl}
Let $M$ be a symmetric space of noncompact type, and let $\ell, \ell' \subseteq \mk{a}$ be one-dimensional linear subspaces. Then the foliations $\mc{F}_\ell$ and $\mc{F}_{\ell'}$ are isometrically congruent if and only if there exists $P \in \Aut(\DD_M)$ whose corresponding orthogonal transformation $\widehat{P}$ of $\mk{a}$ sends $\ell$ onto $\ell'$.
\end{proposition}

Combining Propositions \ref{easypart}, \ref{congruenceFi}, and \ref{congruenceFl} yields a proof of the Main Theorem.

\vspace{0.55em}

\section{Alternative proof of Proposition \ref{congruenceFl}}\label{appendix}

In this section we provide an alternative proof of the more difficult part of Proposition \ref{congruenceFl}, namely that if the foliations $\mc{F}_\ell$ and $\mc{F}_{\ell'}$ are congruent, then $\ell$ and $\ell'$ lie in the same $\Aut(\DD_M)$-orbit. The difficulty here lies in the fact that if the congruence is realized by some $k \in \widetilde{G}$ and even if we assume that $k$ preserves $o$, it may not preserve $\mk{a}$, let alone $\mk{n}$. In their original proof of this statement in the irreducible case, Berndt and Tamaru bypassed this problem by establishing an isomorphism $\mk{s}_\ell \isoto \mk{s}_{\ell'}$ that respects certain natural gradings on these Lie algebras. The main obstacle in this approach is that this isomorphism does not in general come from $\Ad(\widetilde{K})$ and is just an abstract Lie algebra isomorphism. Yet, the authors managed to prove -- purely algebraically -- that the existence of such an isomorphism implies that $\ell$ and $\ell'$ differ by $\Aut(\DD_M)$ (see pp. 9-20 in \cite{berndttamarufoliations}). The proof is quite complicated and involves a case-by-case consideration of all possible irreducible root systems from $A_r$ to $G_2$ and $(BC)_r$.

\vspace{-0.1em}

In their alternative approach taken in \cite{DV-SL_isomparametric_hypersurfaces}, Dom\'{i}nguez-V\'{a}zquez and Sanmart\'{i}n-L\'{o}pez do not relinquish the congruence $k$ and instead use some subtle algebro-geometric arguments and modify $k$ to preserve $\mk{a}$ and $\mk{n}$, which gives them an element of $N_{\widetilde{K}}(\mk{n})$ (and thus $\Aut(\DD_M)$) sending $\ell$ onto $\ell'$. This results in a much shorter proof that works in the reducible case and applies to a more general situation described in Theorem \ref{DV-SM}.

\vspace{-0.1em}

On the principle that two proofs rooted in different sets of ideas are always better than one, we give a proof of Proposition \ref{congruenceFl} alternative to the one in \cite{DV-SL_isomparametric_hypersurfaces}, which is mostly algebraic in nature and is in fact a direct extension of the original proof by Berndt and Tamaru.

\vspace{-0.1em}

Before we begin, given $\ell \subseteq \mk{a}$, note that the solvable Lie algebra $\mk{s}_\ell$ is naturally graded by the height function:
$$
\mk{s}_\ell = \bigoplus_{i = 0}^{m} \mk{s}_\ell^i \hspace{0.7em} \text{with} \hspace{0.7em} \mk{s}_\ell^0 = \mk{a}_\ell = \mk{a} \ominus \ell \hspace{0.7em} \text{and} \hspace{0.7em} \mk{s}_\ell^i = \mk{n}^i = \bigoplus_{\mathrm{ht}(\upalpha) = i} \mk{g}_\upalpha \hspace{0.8em} \text{for} \hspace{0.5em} i \geqslant 1.
$$
In particular, $\mk{s}_\ell^1 = \mk{n}^1 = \bigoplus_{\upalpha \in \Uplambda} \mk{g}_\upalpha$. If we denote $L^k = \bigoplus_{i=k}^m \mk{s}_\ell^i$, it follows from the properties of restricted roots that $[\mk{s}_\ell, \mk{s}_\ell] = L^1 = \mk{n}$ and $[L^1, L^k] = L^{k+1}$ for $k \geqslant 0$. Note that $\mk{s}_\ell$ is completely solvable\footnote{Over $\mathbb{\Cx}$, the notions of solvability and complete solvability coincide by Lie's theorem.}, i.e. $\ad(\mk{s}_\ell) \subset \mk{gl}(\mk{g})$ consists of upper-triangular matrices in a suitable basis for $\mk{s}_\ell$. Indeed, first take a basis for $\mk{s}_\ell^m$, then for $\mk{s}_\ell^{m-1}$, and so on, and then combine all these bases together. Clearly, each $\ad(x), x \in \mk{s}_\ell$, is upper-triangular in the resulting basis for $\mk{s}_\ell$.

\vspace{-0.7em}

\begin{proof}[Proof of Proposition \ref{congruenceFl}]
We will be using the notation established in the proof of Proposition \ref{congruenceFi}. Let $\ell$ and $\ell'$ be one-dimensional subspaces of $\mk{a}$ such that the foliations $\mc{F}_\ell$ and $\mc{F}_{\ell'}$ are congruent by some $k \in \widetilde{G}$. Since all the leaves of $\mc{F}_{\ell}$ are congruent, we may assume that $k \in \widetilde{K}$. First, suppose that $\ell \subseteq \mk{a}_i$ for some $i \in \set{1, \ldots, k}$. Borrowing the notation from the proof of Proposition \ref{product_isometries}, we see that the leaf $S_\ell \ccdot o$ contains $M_{l,o}$ for all $l \ne i$. By Proposition \ref{product_isometries}, the leaf $S_{\ell'} \ccdot o = k(S_\ell \ccdot o)$ must then contain $M_{l,o}$ for all $l \ne j$, where $j \in \set{1, \ldots, k}$ is such that $k(M_{i,o}) = M_{j,o}$. In particular, we must have $\ell' \subseteq \mk{a}_j$ and $M_i \simeq M_j$. Arguing in a similar vein as in the proof of Proposition \ref{congruenceFi}, we see that there is an isomorphism $\DD_{M_i} \isoto \DD_{M_j}$ that extends to an automorphism $P \in \Aut(\DD_M)$ such that $\widehat{P}$ sends $\ell$ onto $\ell'$. Therefore, we may assume that $\ell$ (and hence $\ell'$) does not lie in any of the $\mk{a}_i$'s.

In this case, one can show that $\mk{a}_\ell$ and $\mk{a}_{\ell'}$ are Cartan subalgebras of $\mk{s}_\ell$ and $\mk{s}_{\ell'}$, respectively (see \cite[Lemma 3.3]{berndttamarufoliations}). Recall that both $S_\ell$ and $S_{\ell'}$ are connected and completely solvable (this just means that their Lie algebras are completely solvable). By means of the congruence $k$, we can regard both of these groups as Lie subgroups of $I(S_\ell \ccdot o)$. It follows from \cite{completelysolvable} that any two connected completely solvable transitive Lie groups of isometries of a connected Riemannian manifold are conjugate in the isometry group of that manifold. Therefore, $S_\ell$ and $S_{\ell'}$ are isomorphic, hence so are $\mk{s}_\ell$ and $\mk{s}_{\ell'}$. Let $F \colon \mk{s}_\ell \isoto \mk{s}_{\ell'}$ be an isomorphism.

The first step is to adjust $F$ to make it look nice. Observe that $F(\mk{a}_\ell)$ is a Cartan subalgebra of $\mk{s}_{\ell'}$. Every two Cartan subalgebras of a solvable Lie algebra are conjugate by an inner automorphism, so we may assume $F(\mk{a}_\ell) = \mk{a}_{\ell'}$. Note also that $\mk{n} = [\mk{s}_\ell, \mk{s}_\ell] = [\mk{s}_{\ell'}, \mk{s}_{\ell'}]$, so $F(\mk{n}) = \mk{n}$. Next we modify $F$ to make it into a graded isomorphism. Writing $X^i$ for the $i$-th graded component of a vector $X$, define a map $\mk{s}_\ell \to \mk{s}_{\ell'}, X \mapsto \sum_{i=0}^m (F(X^i))^i$. It is easy to check that this map is a graded Lie algebra isomorphism (see \cite[Theorem 3.4]{berndttamarufoliations} for an argument). We continue to denote it by $F$.

The main idea of the proof is to use the commutator relations between $\mk{a}_\ell$ (or $\mk{a}_{\ell'}$) and $\mk{n}$ to show that $F$ must permute the restricted root subspaces of $\mk{n}$ in a way that induces an automorphism of $\Aut(\DD_M)$. For each pair $\upalpha, \upbeta \in \Uplambda, \upalpha \ne \upbeta$, let $L_{\upalpha\upbeta}$ stand for the hyperplane in $\mk{a}$ consisting of the vectors $Z$ such that the eigenvalue of $\ad(Z)$ on $\mk{g}_\upalpha$ coincides with that on $\mk{g}_\upbeta$. If we write $Z = \sum_{\upgamma \in \Uplambda} Z_\upgamma H^\upgamma$, then $L_{\upalpha\upbeta} = \set{Z \in \mk{a} \mid Z_\upalpha = Z_\upbeta} = \mk{a} \ominus \Rl(H_\upalpha - H_\upbeta)$.

First, we consider the generic choice of $\ell$ such that $\mk{a}_\ell \ne L_{\upalpha\upbeta}$ for any pair $\upalpha, \upbeta \in \Uplambda$. In this case, $\bigcup_{\substack{\upalpha, \upbeta \in \Uplambda \\ \upalpha \ne \upbeta}} (\mk{a}_\ell \cap L_{\upalpha\upbeta})$ is the union of finitely many hyperplanes in $\mk{a}_\ell$, so its complement $\mk{a}_\ell^\circ$ is open and dense in $\mk{a}_\ell$. Pick any $Z \in \mk{a}_\ell^\circ$. By design, all coordinates $Z_\upgamma$ of $Z$ with respect to the basis $(H^\upgamma)_{\upgamma \in \Uplambda}$ of $\mk{a}$ are pairwise distinct. Since $Z_\upgamma$ is the eigenvalue of $\ad(Z)$ on $\mk{g}_\upgamma$, $\ad(Z)$ has the maximal possible number ($= r$) of distinct eigenvalues on $\mk{n}^1$ among all vectors in $\mk{a}$. Now, given $X \in \mk{g}_\upgamma$, one has
$$
[F(Z), F(X)] = F[Z,X] = F(Z_\upgamma X) = Z_\upgamma F(X),
$$
hence $F$ maps the eigenspaces of $\ad(Z)$ in $\mk{n}^1$ onto the eigenspaces of $\ad(F(Z))$ in $\mk{n}^1$ corresponding to the same eigenvalues. Since there are $r$ such eigenvalues, $F$ must permute the root subspaces in $\mk{n}^1$. Formally, there exists a bijection $\widetilde{F} \colon \Uplambda \isoto \Uplambda$ such that
$$
F(\mk{g}_\upgamma) = \mk{g}_{\widetilde{F}(\upgamma)} \; (\forall \; \upgamma \in \Uplambda) \hspace{0.5em} \text{and} \hspace{0.3em} \; F(Z) = \sum_{\upgamma \in  \Uplambda} Z_\upgamma H^{\widetilde{F}(\upgamma)}.
$$
We see that $\ell$ is generic in the sense explained above if and only if $\ell'$ is, and we also have $F(\mk{a}_\ell^\circ) = \mk{a}_{\ell'}^\circ$. Assume for a moment that $\widetilde{F}$ is an automorphism of $\DD_M$. The corresponding orthogonal transformation $T = \widehat{\widetilde{F}}$ of $\mk{a}$ sends $H^\upgamma$ to $H^{\widetilde{F}(\upgamma)}$. Therefore, we have $T(Z) = \sum_{\upgamma \in  \Uplambda} Z_\upgamma H^{\widetilde{F}(\upgamma)}$. We conclude that the restrictions of $T$ and $F$ to $\mk{a}_\ell^\circ$ coincide. Since $\mk{a}_\ell^\circ$ is dense in $\mk{a}_\ell$, these two operators coincide on the whole $\mk{a}_\ell$. But $T \in \O(\mk{a})$, which means that it sends $\ell = \mk{a} \ominus \mk{a}_\ell$ onto $\mk{a} \ominus T(\mk{a}_\ell) = \mk{a} \ominus \mk{a}_{\ell'} = \ell'$, which was to be proved. So we are left to show that $\widetilde{F}$ is indeed an automorphism of $\DD_M$.

Suppose that $\upalpha, \upbeta \in \Uplambda$ are connected by an edge in $\DD_M$, i.e. the angle between them is greater than $\frac{\uppi}{2}$. This is equivalent to asking that $\mk{g}_{\upalpha + \upbeta} = [\mk{g}_\upalpha, \mk{g}_\upbeta] \ne \set{0}$. But then
\begin{equation}\label{adjacency}
\mk{g}_{\widetilde{F}(\upalpha) + \widetilde{F}(\upbeta)} = [\mk{g}_{\widetilde{F}(\upalpha)}, \mk{g}_{\widetilde{F}(\upbeta)}] = [F(\mk{g}_\upalpha), F(\mk{g}_\upbeta)] = F([\mk{g}_\upalpha, \mk{g}_\upbeta]) = F(\mk{g}_{\upalpha + \upbeta}) \ne \set{0},
\end{equation}
which means that $\widetilde{F}(\upalpha)$ and $\widetilde{F}(\upbeta)$ are also connected by an edge in the Dynkin diagram. Applying the same argument to $\widetilde{F}^{-1}$, we see that $\widetilde{F}$ preserves adjacency between the vertices. Consequently, $\widetilde{F}$ permutes the connected components of $\DD_M$: there exists $\upsigma \in S_k$ such that $\widetilde{F}(\DD_{M_i}) = \DD_{M_{\upsigma(i)}}$ for each $i \in \set{1, \ldots, k}$. Note that we then have $F(\mk{n}^1_i) = \mk{n}^1_{\upsigma(i)}$ and, since $\mk{n}^1_i$ generates $\mk{n}_i$ for each $i$, $F(\mk{n}_i) = \mk{n}_{\upsigma(i)}$. We will now show that $\widetilde{F}$ provides an isomorphism between $\DD_{M_i}$ and $\DD_{M_{\upsigma(i)}}$.

To begin with, note that $\DD_{M_i}$ and $\DD_{M_{\upsigma(i)}}$ must have the same number of vertices, and $\widetilde{F}$ preserves the degrees and multiplicities of the vertices. Next, recall that any positive root $\upeta$ can be expressed as $\upgamma_{l_1} + \upgamma_{l_2} + \cdots + \upgamma_{l_s}$, where each summand is a simple root and each partial sum $\upgamma_{l_1} + \cdots + \upgamma_{l_t}, t \in \set{2, \ldots, t-1}$, is also a root. This, together with a computation similar to \eqref{adjacency}, implies that a linear combination $\sum_{j=1}^{r_i} n_j \upalpha_i^j$ is a root in $\Upsigma_i^+$ if and only if $\sum_{j=1}^{r_i} n_j \widetilde{F}(\upalpha_i^j)$ is a root in $\Upsigma_{\upsigma(i)}^+$. Among other things, this implies that $\# \Upsigma_i = \# \Upsigma_{\upsigma(i)}$ and that $\Upsigma_i$ is reduced if and only if $\Upsigma_{\upsigma(i)}$ is. By looking at the list of the Dynkin diagrams of all possible irreducible symmetric spaces of noncompact type (see, for instance, \cite[pp. 336-340]{submanifoldsholonomy}), one deduces that $\DD_{M_i}$ and $\DD_{M_{\upsigma(i)}}$ must be isomorphic (which is also the same as to say that $M_i$ and $M_{\upsigma(i)}$ are isometric or that $\mk{g}_i$ and $\mk{g}_{\upsigma(i)}$ are isomorphic) and one such isomorphism is given by $\widetilde{F}$, which completes the proof for the generic choice of $\ell$.

\begin{remark}
To tell the diagrams $B_r$ and $C_r$ apart, one might need to use the fact that if we denote the two adjacent vertices of nonequal lengths by $\upalpha_{r-1}$ and $\upalpha_r$, the sum $\upalpha_{r-1} + 2\upalpha_r$ is a root for $B_r$ but not for $C_r$.
\end{remark}

We are left to consider the situation when $\mk{a}_\ell = L_{\upalpha\upbeta}$ for some $\upalpha, \upbeta \in \Uplambda, \upalpha \ne \upbeta$, which simply means that $\ell$ is spanned by $H_\upalpha - H_\upbeta$. Since we assume that $\ell$ does not lie in any of the $\mk{a}_l$'s, the roots $\upalpha$ and $\upbeta$ must lie in different components of $\DD_M$, i.e. $\upalpha \in \Uplambda_i, \upbeta \in \Uplambda_j, i \ne j$. In this case, each $L_{\upalpha' \upbeta'}$ with $(\upalpha', \upbeta') \ne (\upalpha, \upbeta)$ intersects $\mk{a}_\ell$ by a hyperplane in $\mk{a}_\ell$. Let $\mk{a}_\ell^\circ \subseteq \mk{a}_\ell$ stand for the complement to the union of all such hyperplanes (for all $(\upalpha', \upbeta') \ne (\upalpha, \upbeta)$). For every $Z \in \mk{a}_\ell^\circ$, $\ad(Z)$ has $r-1$ distinct eigenvalues on $\mk{n}^1$. The only two restricted root subspaces in $\mk{n}^1$ with the same eigenvalue (equal to $Z_\upalpha = Z_\upbeta$) are $\mk{g}_\upalpha$ and $\mk{g}_\upbeta$. As we have already seen, $\ad(F(Z))$ must have the same eigenvalues on $\mk{n}^1$ and $F$ must send the eigenspaces of $\ad(Z)$ onto the corresponding eigenspaces of $\ad(F(Z))$. This implies that there exist some $\upalpha', \upbeta' \in \Uplambda$ such that $\mk{a}_{\ell'} = L_{\upalpha' \upbeta'}$ and thus $\ell'$ is spanned by $H_{\upalpha'} - H_{\upbeta'}$. Since $\ell'$ cannot lie in any of the $\mk{a}_l$'s, we have $\upalpha' \in \Uplambda_{i'}, \upbeta' \in \Uplambda_{j'}, i' \ne j'$. Similarly to what we did in the proof of Proposition \ref{congruenceFi}, note that $S_\ell \ccdot o = (\prod_{l \ne i,j} M_l) \times (S_i \times S_j)_\ell \ccdot o$ (here $(S_i \times S_j)_\ell$ is the connected Lie subgroup of $G_i \times G_j$ corresponding to the Lie subalgebra $(\mk{s}_i \oplus \mk{s}_j) \ominus \ell$) and $S_{\ell'} \ccdot o = (\prod_{l \ne i',j'} M_l) \times (S_{i'} \times S_{j'})_{\ell'} \ccdot o$. Therefore, the congruence $k$ between $\mc{F}_\ell$ and $\mc{F}_{\ell'}$ must send $M_{i,o}$ onto either $M_{i',o}$ or $M_{j',o}$ and $M_{j,o}$ onto the other of the two. Consequently (and slightly informally), the (unordered) pair $(M_i, M_j)$ is isometric to the pair $(M_{i'}, M_{j'})$. This means that there exists an isomorphism of Dynkin diagrams $\DD_{M_i} \sqcup \DD_{M_j} \isoto \DD_{M_{i'}} \sqcup \DD_{M_{j'}}$, which extends easily -- just like we did at the end of the proof of Proposition \ref{congruenceFi} -- to an automorphism $P$ of $\DD_M$. We will prove that all four of the spaces $M_i, M_j, M_{i'}, M_{j'}$ have rank 1 and are thus isometric to hyperbolic spaces. If this is the case, then $P$ sends $\Uplambda_i \sqcup \Uplambda_j = \set{\upalpha, \upbeta}$ onto $\Uplambda_{i'} \sqcup \Uplambda_{j'} = \set{\upalpha', \upbeta'}$, hence the corresponding orthogonal transformation $\widehat{P}$ of $\mk{a}$ sends $H_\upalpha - H_\upbeta \in \ell$ to $\pm (H_{\upalpha'} - H_{\upbeta'}) \in \ell'$, which will complete the proof.

It follows from our argument involving the eigenspaces of $\ad(Z)$ and $\ad(F(Z))$ that $F$ establishes a bijection
\begin{equation}\label{bijection}
    \set{\mk{g}_{\upalpha_1}, \ldots, \widehat{\mk{g}}_\upalpha, \ldots, \widehat{\mk{g}}_\upbeta, \ldots, \mk{g}_{\upalpha_r}, \mk{g}_\upalpha \oplus \mk{g}_\upbeta} \isoto \set{\mk{g}_{\upalpha_1}, \ldots, \widehat{\mk{g}}_{\upalpha'}, \ldots, \widehat{\mk{g}}_{\upbeta'}, \ldots, \mk{g}_{\upalpha_r}, \mk{g}_{\upalpha'} \oplus \mk{g}_{\upbeta'}},
\end{equation}
sending subspaces on the left isomorphically onto the corresponding subspaces on the right (here a hat over a subspace means that it is omitted from the list). We have to consider two cases.

\textit{Case 1:} $F(\mk{g}_\upalpha \oplus \mk{g}_\upbeta) = \mk{g}_{\upgamma}$ for some $\upgamma \ne \upalpha', \upbeta'$. We first prove two Lie-theoretic lemmas of separate interest.

Observe that the adjoint action of the compact Lie group $Z_{\widetilde{K}}(\mk{a})$ on $\mk{g}$ preserves the restricted root space decomposition. Thus, this group acts orthogonally on each restricted root subspace.

\begin{lemma}\label{helgason_lemma}
For every $\upmu \in \Upsigma$ of multiplicity greater than one, the action $Z_{\widetilde{K}}(\mk{a}) \curvearrowright \mk{g}_\upmu$ is transitive on the unit sphere in $\mk{g}_\upmu$. It is even transitive when restricted to the connected Lie subgroup $K_0 \subseteq Z_{\widetilde{K}}(\mk{a})$ corresponding to the Lie subalgebra $\mk{k}_0$.
\end{lemma}

\begin{proof}
If $\upmu$ is a simple root or the double of a simple root, this is the content of Problem D.2 in \cite[Ch. III]{helgason_analysis} (see p. 585 for a solution). But it is a well-known fact that every root in any root system is either a simple root or the double of one under a suitable choice of a Weyl chamber. The second statement follows trivially from the first one, because $K_0$ is the connected component of the identity in $Z_{\widetilde{K}}(\mk{a})$ and the unit sphere in $\mk{g}_\upmu$ is connected.
\end{proof}

\begin{lemma}\label{pairing_lemma}
Let $\upmu, \upnu \in \Upsigma$ be any two roots such that $\upmu + \upnu$ is also a root (e.g. they can be simple roots connected by an edge in the Dynkin diagram $\DD_M$). Then the pairing $\mk{g}_\upmu \times \mk{g}_\upnu \to \mk{g}_{\upmu + \upnu}$ given by the Lie bracket is nondegenerate\footnote{An analogous result in the complex semisimple case is well known and is just a reformulation of the fact that $[\mk{g}_\upmu, \mk{g}_\upnu] = \mk{g}_{\upmu + \upnu}$, for all the root subspaces are one-dimensional over $\Cx$.}. In other words, for every $X \in \mk{g}_\upmu$ (resp., $Y \in \mk{g}_\upnu$), there exists $Y \in \mk{g}_\upnu$ (resp., $X \in \mk{g}_\upmu$) such that $[X,Y] \ne 0$.
\end{lemma}

\begin{proof}
Assume the converse: there exists $X \in \mk{g}_\upmu$ such that $[X, \mk{g}_\upnu] = \set{0}$. For each $k \in Z_{\widetilde{K}}(\mk{a})$, we have $[kX, \mk{g}_\upnu] = [kX, k(\mk{g}_\upnu)] = k[X, \mk{g}_\upnu] = \set{0}$. But, according to Lemma \ref{helgason_lemma}, every vector in $\mk{g}_\upmu$ is a multiple of $kX$ for a suitable choice of $k$. Therefore, $[\mk{g}_\upmu, \mk{g}_\upnu] = \set{0}$, which contradicts the standard fact that $[\mk{g}_\upmu, \mk{g}_\upnu] = \mk{g}_{\upmu + \upnu} \ne \set{0}$.
\end{proof}

Now, assume there is $\upgamma' \in \Uplambda$ connected to $\upgamma$ by an edge in $\DD_M$. Note that $F^{-1}(\mk{g}_{\upgamma'})$ must lie within a single restricted root subspace, which we denote by $\mk{g}_{\upgamma''}, \upgamma'' \in \Uplambda$. By Lemma \ref{pairing_lemma}, for any nonzero $X \in \mk{g}_\upalpha$, there exists $Y \in \mk{g}_{\upgamma'}$ with $[F(X), Y] \ne 0$, which implies that $[X, F^{-1}(Y)] \ne 0$, hence $\upgamma''$ must be connected to $\upalpha$ in $\DD_M$. But we can also apply this argument to a nonzero $X \in \mk{g}_\upbeta$ to deduce that $\upgamma''$ must be connected to $\upbeta$ as well. Since $\upalpha$ and $\upbeta$ lie in different connected components of $\DD_M$, we arrive at a contradiction. Thus, the connected component of $\DD_M$ containing $\upgamma$ consists of nothing but $\upgamma$. But then the same must hold for $\upalpha$ and $\upbeta$: $\Uplambda_i = \set{\upalpha}, \Uplambda_j = \set{\upbeta}$. Indeed, otherwise we would have $\upgamma'$ connected to, say, $\upalpha$, meaning that $[\mk{g}_\upalpha, \mk{g}_{\upgamma'}] \ne \set{0}$, whereas $F([\mk{g}_\upalpha, \mk{g}_{\upgamma'}]) = [F(\mk{g}_\upalpha), F(\mk{g}_{\upgamma'})] \subseteq [\mk{g}_\upgamma, F(\mk{g}_{\upgamma'})] = \set{0}$. We conclude that $\Uplambda_i$ and $\Uplambda_j$ are singletons, i.e. $M_i$ and $M_j$ have rank 1. The same considerations can be applied to $F^{-1}$ to deduce that $M_{i'}$ and $M_{j'}$ are also of rank 1. This finishes the proof in Case 1.

\textit{Case 2:} $F(\mk{g}_\upalpha \oplus \mk{g}_\upbeta) = \mk{g}_{\upalpha'} \oplus \mk{g}_{\upbeta'}$. It suffices to show that $F(\mk{g}_\upalpha)$ is one of the subspaces $\mk{g}_{\upalpha'}, \mk{g}_{\upbeta'}$, while $F(\mk{g}_\upbeta)$ is the other one. Indeed, if this is the case, then we get a permutation $\widetilde{F}$ of $\Uplambda$, which -- as we have already seen in the proof for a generic choice of $\ell$ -- must be an automorphism of $\DD_M$ whose corresponding orthogonal transformation of $\mk{a}$ maps $\ell$ onto $\ell'$. Arguing by contradiction, we may assume, without loss of generality, that the projections of $F(\mk{g}_\upalpha)$ in both $\mk{g}_{\upalpha'}$ and $\mk{g}_{\upbeta'}$ are nonzero. We also do not lose generality by assuming that the projection of $F(\mk{g}_\upbeta)$ in $\mk{g}_{\upbeta'}$ is nonzero. By invoking Lemma \ref{pairing_lemma}, we see, just as we did for $\upgamma$ in Case 1, that $\Uplambda_{j'} = \set{\upbeta'}$. There might be two possibilities.

\textit{Subcase 2.1:} $F(\mk{g}_\upbeta) \subseteq \mk{g}_{\upbeta'}$. In this case, we can write $\mk{g}_\upalpha = U \oplus V$, where $F(U) = \mk{g}_{\upalpha'}$ and $F(V) \oplus F(\mk{g}_\upbeta) = \mk{g}_{\upbeta'}$. Mimicking the arguments from the proof in Case 1, we immediately see that $\Uplambda_j = \set{\upbeta}$. If there exists $\upgamma \in \Uplambda$ connected to $\upalpha$ in $\DD_M$, then, by Lemma \ref{pairing_lemma}, for any nonzero $X \in V$, there exists $Y \in \mk{g}_\upgamma$ such that $[X,Y] \ne 0$. But then $[F(X), \mk{g}_{\upbeta'}] \ne \set{0}$, which implies that $F(\mk{g}_\upgamma)$ must be a root subspace whose corresponding root is connected to $\upbeta$, which is a contradiction. Therefore, $\Uplambda_i = \set{\upalpha}$. Arguing in the same fashion, we deduce that $\Uplambda_{i'} = \set{\upalpha'}$, so all four of $M_i, M_j, M_{i'}, M_{j'}$ are of rank 1 and we are done.

\textit{Subcase 2.2}: The projection of $F(\mk{g}_\upbeta)$ in $\mk{g}_{\upalpha'}$ is nonzero. Just like we did for $\Uplambda_{j'}$, we see that $\Uplambda_{i'} = \set{\upalpha'}$. Arguing in a similar manner, we deduce that $\Uplambda_i = \set{\upalpha}$ and $\Uplambda_j = \set{\upbeta}$. Once again, all four of $M_i, M_j, M_{i'}, M_{j'}$ are of rank 1. This completes the proof of Subcase 2.2 and thus Proposition \ref{congruenceFl}.
\end{proof}

\printbibliography

@article{berndttamarufoliations,
    author = {J. Berndt and H. Tamaru},
    title = {Homogeneous codimension one foliation on noncompact symmetric spaces},
    year = {2003},
    journal = {Journal of Differential Geometry},
    publisher = {International Press of Boston},
    volume = {63},
    number = {1},
    pages = {1-40}
}

@article{hyperpolarfoliations,
    author = {J. Berndt and J. C. D\'{i}az-Ramos and H. Tamaru},
    title = {Hyperpolar homogeneous foliations on symmetric spaces of noncompact type},
    year = {2010},
    journal = {Journal of Differential Geometry},
    publisher = {International Press of Boston},
    volume = {86},
    number = {2},
    pages = {191-236}
}

@book{helgason,
    author = {S. Helgason},
    title = {Differential geometry, Lie groups, and symmetric spaces},
    year = {1978},
    publisher = {Academic Press},
}

@book{knapp,
    author = {A. Knapp},
    title = {Lie Groups Beyond an Introduction},
    year = {2002},
    edition = {2nd edn},
    publisher = {Birkh\"{a}user},
    series = {Progress in Mathematics},
    volume = {140}
}

@book{submanifoldsholonomy,
    title = {Submanifolds and Holonomy},
    author = {J. Berndt and S. Console and C. E. Olmos},
    year = {2016},
    edition = {2nd edn},
    publisher = {CRC Press}
}

@article{completelysolvable,
    title = {Conjugacy of polar factorizations of {Lie} groups},
    author = {D. V. Alekseevski\v{i}},
    year = {1971},
    journal = {Math. USSR Sbornik},
    volume = {13},
    number = {1}
}

@book{kobayashi_nomizu_I,
    title = {Foundations of Differential Geometry},
    author = {S. Kobayashi and K. Nomizu},
    year = {1996},
    publisher = {Wiley Classics Library},
    volume = {I}
}

@book{borel_ji,
    author = {A. Borel and L. Ji},
    title = {Compactifications of Symmetric and Locally Symmetric Spaces},
    year = {2006},
    publisher = {Birkh\"{a}user Basel}
}

@misc{mypaper,
      title={Classification of homogeneous hypersurfaces in some noncompact symmetric spaces of rank two}, 
      author={Ivan Solonenko},
      year={2021},
      eprint={2111.05280},
      archivePrefix={arXiv},
      primaryClass={math.DG}
}

@misc{DV-SL_isomparametric_hypersurfaces,
    title={Isoparametric hypersurfaces in symmetric spaces of non-compact type and higher rank},
    author={Miguel Dom\'{i}nguez-V\'{a}zquez and Victor Sanmart\'{i}n-L\'{o}pez},
    year={2021},
    eprint={2109.03850},
    archivePrefix={arXiv},
    primaryClass={math.DG}
}

@book{helgason_analysis,
    title = {Geometric analysis on symmetric spaces},
    author = {Sigurdur Helgason},
    year = {2008},
    publisher = {American Mathematical Society},
    edition = {2nd edn},
    series = {Mathematical Surveys and Monographs},
    volume = {39}
}

\end{document}